\documentclass[12pt,a4paper]{article}
\usepackage[english]{babel}
\usepackage[utf8]{inputenc}
\usepackage{amsfonts,amssymb,amsmath}
\usepackage{theorem}
\usepackage{graphicx}
\usepackage[margin=25mm]{geometry}

\def\ts{\textstyle}

\def\e{\varepsilon}
\def\g{\gamma}
\def\R#1{\ensuremath{\mathbb{R}^{#1}}}
\def\S#1{\ensuremath{\mathbb{S}^{#1}}}
\def\H#1{\ensuremath{\mathbb{H}^{#1}}}

\def\r#1#2{\ensuremath{\mathbb{R}^{#1}_{#2}}}

\def\M{\ensuremath{\mathbb{M}}}

\def\Xfrak{\mathfrak{X}}
\def\<{\left<}
\def\>{\right>}
\def\norma#1{|\kern-.3ex|#1|\kern-.3ex|}
\def\ol#1{\overline{#1}}
\def\bnabla{\overline\nabla}

\newenvironment{proof}{\ignorespaces\par\noindent\textbf{Proof.}\quad}{\hfill$\Box$\par}

\def\bfv{\mathbf{v}}

\def\bfn{\mathbf{n}}

\def\d{\text{d}}
\def\pc{\ensuremath{p_0}}

\DeclareMathOperator{\Con}{Con}

{\theorembodyfont{\itshape}
\newtheorem{theo}{Theorem}
\newtheorem{prop}[theo]{Proposition}
\newtheorem{coro}[theo]{Corollary}

} {\theorembodyfont{\upshape}

\newtheorem{defi}{Definition}
\newtheorem{exam}{Example}
}

\newcommand{\displayskip}[1][1]{%
 \abovedisplayskip=#1\abovedisplayskip
 \belowdisplayskip=#1\belowdisplayskip
 \abovedisplayshortskip=#1\abovedisplayshortskip
 \belowdisplayshortskip=#1\belowdisplayshortskip
 }

\title{Concircular hypersurfaces and concircular helices in space forms}
\author{\normalsize Pascual Lucas\footnote{Corresponding author.\newline \hspace*{17pt}E-mail
addresses: plucas@um.es and yagues1974@hotmail.com}\and\normalsize
José Antonio Ortega-Yagües}
\date{\normalsize Departamento de Matemáticas, Universidad de Murcia\\
       Campus de Espinardo, 30100 Murcia SPAIN\\[10mm] \today}

\begin{document}
\displayskip[.7]

\maketitle

\begin{abstract}
In this paper, we find a full description of concircular hypersurfaces in space forms as a special family of ruled hypersurfaces. We also characterize concircular helices in 3-dimensional space forms by means of a differential equation involving the concircular factor and their curvature and torsion, and we show that the concircular helices are precisely the geodesics of the concircular surfaces.
\end{abstract}
\medskip

\medskip


\section{Introduction}
\label{s:intro}

Generalized helices, slant helices and rectifying curves are well known examples of curves satisfying a certain condition with respect to a special vector field. Generalized helices are defined by the property that their tangents make a constant angle with a fixed direction; slant helices are defined by the property that their principal normals make a constant angle with a constant vector field, \cite{IT1}; and rectifying curves are defined as the curves whose position vector is orthogonal to its principal normal vector field (i.e. the position vector lies in the rectifying plane), \cite{Chen03}. Moreover, these curves are characterized as the geodesics of some special ruled surfaces: generalized helices in cylinders, slant helices in helix surfaces, \cite{LO16a}, and rectifying curves in conical surfaces, \cite{Chen17b}. Motivated by these examples of curves and surfaces, the authors in \cite{LO21} have extended the above conditions, and have introduced the notion of concircular submanifold in the Euclidean space $\R{n}$. In particular, they characterize concircular helices in $\R3$ by means of a differential equation involving their curvature and torsion. Moreover, they also find a full description of concircular surfaces in $\R3$ as a special family of ruled surfaces and characterize the concirculares helices in $\R3$ as the geodesics of these surfaces.

In this paper, we generalize the results obtained in \cite{LO21} to space forms of nonzero constant curvature. Recall that a vector field $V\in\Xfrak(M)$ on a Riemannian manifold $M$, with Levi-Civita connection $\nabla$, is said to be \emph{concircular} if $\nabla V=\mu I$, where $\mu\in\mathcal{C}^\infty(M)$ is a differentiable function called the \emph{concircular factor}, (\cite{Fia39}, \cite{Yan40}, \cite{Kim82}). We denote by $\Con(M)$ the set of concircular vector fields of $M$. The following definition extends the one given in \cite{LO21}.
\begin{defi}\label{concsubm}
Let $\M^n(C)$ be an $n$-dimensional space form of constant curvature $C$. A submanifold $M^m\subset\M^n(C)$ is said to be a \emph{concircular submanifold} if there exists a concircular vector field $V\in\Con(\M^n(C))$ (called the axis of $M^m$) such that $\<\bfn,V\>$ is a constant function along $M^m$, $\bfn$ being any unit vector field in the first normal space of $M^m$.
\end{defi}
In the particular case of a hypersurface, $M^{n-1}$ is said to be a concircular hypersurface (with axis $V$) if $\<N,V\>$ is a constant function along $M^{n-1}$, $N$ being a unit normal vector field. Another very interesting case appears when $m=1$: a (non geodesic) unit speed curve $\gamma$ in $\M^n(C)$ is said to be a \emph{concircular helix} (with axis $V$) if $\<N_\gamma,V\>$ is a constant function along $\gamma$, $N_\gamma$ being the principal normal vector field of $\gamma$.

This paper is organized as follows. In section \ref{s:setup} we characterize concircular vector fields in $\M^n(C)$, see theorem \ref{t1}. In section \ref{s:hypersurf} we present several properties of concircular hypersurfaces in $\M^n(C)$, see propositions \ref{propcs1} and \ref{propconc2}, and we finish this section with the characterization of all concircular hypersurfaces in $\M^n(C)$, see theorem \ref{sctheo}.
Section \ref{s:helices} contains a characterization of all concircular helices in $\M^3(C)$, see proposition \ref{prophgm3} and theorem \ref{teoconc1}. Finally, section \ref{s:geodesics} contains the characterization of geodesics curves of concircular surfaces, see proposition \ref{intp1}, and this characterization is used to show that concircular helices in $\M^3(C)$ can be described as the geodesics of the concircular surfaces, see theorem \ref{th1}.

\section{Concircular vector fields in space forms}
\label{s:setup}

Let $\M^n(C)$ denote the $n$-dimensional space form of nonzero constant curvature $C$. Then $\M^n(C)$ stands for a sphere $\S n\subset\R{n+1}$ or a hyperbolic space $\H n\subset\r{n+1}1$ according to $C>0$ or $C<0$, respectively. Put $C=\e/R^2$, with $\e=(-1)^\nu$, where $\nu\in \{0,1\}$ is the index of the ambient space $\r{n+1}\nu$ that contains $\M^n(C)$. $\M^n(C)$ can be described as follows,
\[
\M^n(C)=\{p=(x_1,x_2,\ldots,x_{n+1})\in\R{n+1}_\nu\;|\;\<p,p\>=1/C\},
\]
where as usual $\R{n+1}_\nu$ is the space $\R{n+1}$ endowed with the flat metric
\[
\<,\>=\e\,\d x_1^2+\d x_2^2+\cdots+\d x_{n+1}^2,
\]
$(x_1,x_2,\ldots,x_{n+1})$ being the usual rectangular coordinates of $\R{n+1}$.

Let us write $\nabla^0$ and $\bnabla$ to denote the Levi-Civita connections of $\R{n+1}_\nu$ and $\M^n(C)$, respectively. If $\phi:\M^n(C)\to\r{n+1}\nu$ denotes the usual isometric immersion (the position vector), then the Gauss formula is
\begin{equation}\label{gauss}
  \nabla^0_XY=\bnabla_XY-C\<X,Y\>\phi,
\end{equation}
for any vector fields $X$ and $Y$ tangent to $\M^n(C)$.

Given a point $p\in\M^n(C)$ and a unit vector $w\in T_p\M^n(C)$, the exponential map $\exp_p$ is given by
\begin{equation}\label{37.1}
\exp_{p}(tw)=f\Big(\frac{t}{R}\Big)p+Rg\Big(\frac{t}{R}\Big)w,
\end{equation}
where functions $f$ and $g$ are given by $f(t)=\cos t$ and $g(t)=\sin t$ when $C>0$,
or $f(t)=\cosh t$ and $g(t)=\sinh t$ when $C<0$. Note that $f^2+\e g^2=1$, $f'=-\e g$ and $g'=f$.

The following result characterizes the concircular vector fields.

\begin{theo}\label{t1}
A vector field $V\in\frak X(\M^n(C))$ is concircular if and only if $V$ is the tangential part of a constant vector field $\pc$ in $\R{n+1}_\nu$. Moreover, if $\mu$ is the concircular factor of $V$, then $V= \pc+\mu\phi$, where $\mu=-C\<\pc,\phi\>$.
\end{theo}
\begin{proof}
The curvature tensor of $\M^n(C)$ is given by
\[
R_{XY}Z=\bnabla_{[X,Y]}Z-\bnabla_X(\bnabla_YZ)+\bnabla_Y(\bnabla_XZ).
\]
Then, if $V$ is a concircular vector field with concircular factor $\mu$, we have
\begin{equation}\label{rxvv}
R_{XV}V=V(\mu)X-X(\mu)V.
\end{equation}
On the other hand, since $\M^n(C)$ is a space of constant curvature $C$, its curvature tensor is given by
\begin{equation}\label{tensor curvatura c}
R_{XV}V=C\{\<V,X\>V-\<V,V\>X\}.
\end{equation}
By assuming that $X$ and $V$ are two linearly independent vector fields, from (\ref{rxvv}) and (\ref{tensor curvatura c}) we get
$-C\<V,X\>=X(\mu)$ and $-C\<V,V\>=V(\mu)$, and therefore
\begin{equation}\label{vgrad}
-CV=\nabla\mu.
\end{equation}
Take the vector field $\psi=V-\mu\phi$, then
\[
\nabla^0_X\psi=\bnabla_XV-C\<X,V\>\phi-X(\mu)\phi-\mu X.
\]
From here and again (\ref{vgrad}) we get $\psi$ is constant, and so there exists a constant vector field $\pc \in\R{n+1}_\nu$ such that
\begin{equation}\label{vtangente}
\pc = V-\mu\phi,\quad\text{with }\mu=-C\<\pc ,\phi\>.
\end{equation}

Conversely, let $V=\{\pc \}^\top$ be the tangential part of a constant vector in $\R{n+1}_\nu$. Then we have (\ref{vtangente}), and by derivating there we get $0=\bnabla_XV-C\<X,V\>\phi-X(\mu)\phi-\mu X$,
where $X$ is any tangent vector field in $\M^n(C)$.
Hence $\bnabla_XV=\mu X$ for any $X$, so that $V$ is a concircular vector field with concircular factor $\mu$.
\end{proof}

As a consequence of (\ref{vgrad}) we have the following result.
\begin{coro}
In a space form $\M^n(C)$ of nonzero curvature $C$, the concircular factor is a nonconstant function.
\end{coro}

\begin{prop}
The set $\Con(\M^n(C))$ of all concircular vector fields of $\M^n(C)$ is a real vector space of dimension $n+1$.
\end{prop}
\begin{proof}
It is sufficient to show that each concircular vector field $V$ is determined by one, and only one, constant vector $\pc$. Let us suppose that $V$ is determined by two constant vectors $\pc$ and $q_0$, that is, $V=\pc+\mu_1\phi$ and $V=q_0+\mu_2\phi$, for certain differentiable functions $\mu_1$ and $\mu_2$. Then $\pc-q_0= (\mu_2-\mu_1)\phi$, and so $\mu_1=\mu_2$ and $\pc=q_0$.
\end{proof}

\section{Concircular hypersurfaces}
\label{s:hypersurf}

To begin with, let us show some examples of concircular hypersurfaces in $\M^n(C)$.

\begin{exam}\label{ex-umbil}
A \emph{totally umbilical hypersurface} $Q^{n-1}(c)$ in $\M^n(C)$ can be obtained as the intersection $\M^n(C)\cap H(\pc)$, where $H(\pc)$ is a hyperplane in $\r{n+1}\nu$ orthogonal to a constant vector $\pc\in\r{n+1}\nu$. If $N$ denotes the unit vector field normal to $Q^{n-1}(c)$ in $\M^n(C)$, then $\<N,V\>=\<N,\pc\>$, where $V=\pc-C\<\pc,\phi\>\phi$. By derivating here, we get $X\<N,\pc\>=\<-AX,\pc\>=0$, for any vector field $X$ tangent to $Q^{n-1}(c)$, where $A$ denotes the shape operator associated to $N$. Hence $\<N,V\>$ is constant and so $Q^{n-1}(c)$ is a concircular hypersurface. We will say a totally umbilical hypersurface is a \emph{trivial concircular hypersurface}.
\end{exam}

\begin{exam}\label{ex-conic}
A \emph{conical hypersurface} $M$ in $\M^n(C)$ (with vertex at $\pc\in\M^n(C)$) can be described as follows. Let $P^{n-2}$ be an $(n-2)$-dimensional hypersurface in the unit hypersphere $S^{n-1}(1)$ of the tangent space $T_{\pc}\M^n(C)$. For $\epsilon>0$ sufficiently small, the map $\Psi:P^{n-2}\times(-\epsilon,\epsilon)\to\M^n(C)$ given by
\[
\Psi(v,t)=\exp_{p_0}(tv)=f\Big(\frac tR\Big)p_0+Rg\Big(\frac tR\Big)v,
\]
defines an immersion. The image $M=\Psi(P^{n-2}\times (-\epsilon,\epsilon))$ is said to be a conical hypersurface in $\M^n(C)$ (see \cite{LO14a,LO15} in the case $n=3$).
We can identify in a natural way $P^{n-2}$ with $P^{n-2}\times\{0\}$ and $P^{n-2}\times (-\epsilon,\epsilon)$ with $M$, and then the unit vector field normal to $M$ in $\M^n(C)$ is given without loss of generality by $N(v,t)=\eta(v)$, $\eta$ being the unit vector field normal to $P^{n-2}$ in $S^{n-1}(1)$. Hence $\<N,V\>=0$, for $V=\pc-C\<\pc,\phi\>\phi$, showing that $M$ is a concircular hypersurface.
\end{exam}

Before addressing the characterization of the concircular hypersurfaces, we will present a couple of results.

\begin{prop}\label{propcs1}
Given a hypersurface $M\subset\M^n(C)$, then there exists a concircular vector field parallel to its normal vector field along $M$ if and only if $M$ is a totally umbilical hypersurface in  $\M^n(C)$.
\end{prop}
\begin{proof}
Suppose there exists a concircular vector field $V$ such that $V|_M=\lambda\,N$, for a nonzero differentiable function $\lambda$. Then we get
$\mu\,X=X(\lambda)\,N-\lambda\,AX$,
for any vector field $X$ tangent to $M$, $A$ being the shape operator associated to $N$. From that equation we have that $\lambda$ is a nonzero constant and
$AX=-(\mu/\lambda)\,X$.
Now, from (\ref{vgrad}) we get $X(\mu)=0$ for any tangent vector field $X$, and so $M$ is a totally umbilical hypersurface.

Conversely, if $M$ is a totally umbilical hypersurface then there exists a constant $m$ such that
$\bnabla_XN=-AX=-m\,X$, and then the vector field $N+m\phi$ is constant along $M$. This shows that $N$ is collinear with a concircular vector field along $M$.
\end{proof}
\medskip

Let $M\subset\M^n(C)$ be a nontrivial concircular hypersurface with axis $V$, and write  $\<V,N\>=\lambda$, $\lambda$ being a constant. By decomposing $V$ in its tangential and normal components we have
\begin{equation}\label{scT}
V|_M=\alpha\,T+\lambda\,N,
\end{equation}
where $T$ is a unit vector field tangent to $M$ and $\alpha\neq0$ (otherwise, $M$ would be a trivial concircular hypersurface).

\begin{prop}\label{propconc2}
The integral curves of $T$ are geodesics in $\M^n(C)$.
\end{prop}
\begin{proof}
By derivating (\ref{scT}) and using the Gauss and Weingarten equations we obtain
\begin{equation}\label{sceq1}
\mu X=X(\alpha)T+\alpha\nabla_XT+\alpha\sigma(X,T)-\lambda AX,
\end{equation}
where $\nabla$ and $\sigma$ denote the Levi-Civita connection and the second fundamental form of $M$. From here we get
$\sigma(X,T)=0$,
for any tangent vector field $X$, or equivalently
\begin{equation}\label{scAT}
AT=0.
\end{equation}
By putting $X=T$ in (\ref{sceq1}) and using (\ref{scAT}), we obtain
\begin{equation}\label{sctang}
\nabla_TT=0.
\end{equation}
From here and (\ref{scAT}) we deduce the result.
\end{proof}
\bigskip

\def\pcc{\epsilon}
In what follows, we will characterize nontrivial concircular hypersurfaces in $\M^n(C)$.
Let $Q^{n-1}(c)=H(\pc)\cap\M^n(C)$ be a totally umbilical hypersurface of constant curvature $c$, where $H(\pc)$ is a hyperplane in $\R{n+1}_\nu$ orthogonal to a unit vector $\pc\in\R{n+1}_\nu$.
Let $P^{n-2}$ be a hypersurface of $Q^{n-1}(c)$, and denote by $\eta_1$ and $\eta_2$ the unit vector fields normal to $P^{n-2}$ in $Q^{n-1}(c)$ and normal to $Q^{n-1}(c)$ in $\M^n(C)$, respectively. For a real number $a$, define the unit vector field $W_a(p)=\cos(a)\;\eta_1(p)+\sin(a)\;\eta_2(p)$, where $p\in P^{n-2}$. The map $\Psi_a:P^{n-2}\times I_0\to\M^n(C)$ given by
\[
\Psi_a(p,z)=\exp_p(zW_a(p))=f\Big(\frac zR\Big)p+Rg\Big(\frac zR\Big)W_a(p),
\]
defines an immersion, for an enough small interval $I_0$ around the origin. Let $M$ denote the ruled hypersurface in $\M^n(C)$ given by $\Psi_a(P^{n-2}\times I_0)$. We can identify, in a natural way, $P^{n-2}$ with $\Psi_a(P^{n-2}\times\{0\})$ and $P^{n-2}\times I_0$ with $M$. Without loss of generality, we can assume that the unit vector field normal to $M$ in $\M^n(C)$ is given by
\[
N(p,z)=-\sin(a)\;\eta_1(p)+\cos(a)\;\eta_2(p).
\]
Since $P^{n-2}\subset H(p_0)\cap\M^n(C)$, it is not difficult to see that $p_0\in\text{span}\{\eta_2(p),\phi(p)\}$, for any $p\in P^{n-2}$, so we can write $p_0=A\eta_2+B\phi$, for two differentiable funcions $A$ and $B$. By taking derivative here, and using that $p_0$ is constant, we get that $A$ and $B$ are also constants.
Let us consider the concircular vector field in $\M^n(C)$ given by $V=\pc-C\<\pc,\phi\>\phi$.
It is not difficult to see that $\<V,N\>=\<\pc,N\>=A\cos(a)$ is constant, and so $M$ is a concircular hypersurface.

Note that a hypersurface $M$ in $\M^n(C)$ is concircular if and only if there exists a point $p_0\in\R{n+1}_\nu$ such that $\<N,p_0\>$ is constant, $N$ being the unit normal vector field of $M$ in $\M^n(C)$.

The main result of this section is the following. We will show that every nontrivial concircular hypersurface in $\M^n(C)$ can be obtained by the construction described above.

\begin{theo}\label{sctheo}
Let $M\subset\M^n(C)$ be a nontrivial concircular hypersurface with axis $V$. Then there exists a hypersurface $P^{n-2}$ in a totally umbilical hypersurface $Q^{n-1}(c)\subset\M^n(C)$, such that $M$ can be locally described by
\begin{equation}\label{teoparametrization}
\Psi_a(p,z)=\exp_p(zW_a(p))=f\Big(\frac zR\Big)p+Rg\Big(\frac zR\Big)W_a(p),
\end{equation}
where $a\in\R{}$ and $(p,z)\in P^{n-2}\times I_0$, $I_0$ being an interval around the origin.
\end{theo}
\begin{proof}
Let us suppose that the axis $V$ of $M$ is given by $V=\pc-C\<\pc,\phi\>\phi$, for a constant vector $\pc\in\R{n+1}_\nu$, and assume $\<V,N\>=\lambda$.
Pick a point $q$ in $M$ and let $Q^{n-1}(c)=H(\pc)\cap\M^n(C)$ be a totally umbilical hypersurface containing $q$, $H(\pc)$ being a hyperplane in $\R{n+1}_\nu$ orthogonal to $\pc$.
Since $M$ is a nontrivial concircular hypersurface, then there is an $(n-2)$-dimensional submanifold $P^{n-2}\subset Q^{n-1}(c)\cap M$ with $q\in P^{n-2}$.

Let $T$ be the unit vector field tangent to $M$ which is collinear with the tangential component of $V$. From proposition \ref{propconc2} we deduce that there exists a neighborhood $U(q)$ of $q$ in $M$ given by
\[
\ts
U(q)=\{f\left(\frac{z}{R}\right)p+R\, g\left(\frac{z}{R}\right)T(p)\;|\;p\in U_1(q),\ z\in(-\e_2,\e_2)\},
\]
where $U_1(q)$ is an neighborhood of $q$ in $P^{n-2}$.
Since $T$ is orthogonal to $P^{n-2}$, but tangent to $M$, then there exist a differentiable function $a\in{\mathcal C}^\infty(P^{n-2})$  such that
\[
T(p)=\cos(a(p))\;\eta_1(p) +\sin(a(p))\;\eta_2(p),
\]
and, up to the sign, the unit normal vector field $N$ along $P^{n-2}$ is given by
\[
N(p)=-\sin(a(p))\;\eta_1(p) +\cos(a(p))\;\eta_2(p).
\]
Then we have $\lambda=\<V,N\>(p)=\<\pc,N\>(p)=A\cos(a(p))$ is constant and so $a$ is a constant function. Hence the open set $U(q)$ can be rewritten as in (\ref{teoparametrization}).
\end{proof}
\bigskip

In example \ref{ex-conic} we have seen that the conical hypersurfaces are concircular hypersurfaces associated to the constant value $\lambda=0$. Now, we will prove that these hypersurfaces are the only ones that satisfy this property.

\begin{prop}
Let $M\subset\M^n(C)$ be a concircular hypersurface, with axis $V$, such that $\<V,N\>=0$, $N$ being the unit normal vector field. Then $M$ is a conical hypersurface.
\end{prop}
\begin{proof}
From theorem \ref{sctheo} we know that $M$ can be locally described by
\[
\Psi_a(p,z)=f\Big(\frac zR\Big)p+Rg\Big(\frac zR\Big)W_a(p),\quad W_a(p)=\cos(a)\;\eta_1(p) +\sin(a)\;\eta_2(p),
\]
where $(p,z)\in P^{n-2}\times I_0$, for certain $P^{n-2}\subset Q^{n-1}(c)\subset\M^n(C)$, $Q^{n-1}(c)=H(p_0)\cap\M^n(C)$ and $I_0\subset\R{}$. Suppose that the interval $I_0$ is the largest possible interval.

The point $p_0$ can be written as $p_0=A\eta_2+B\phi$, for certain constants $A$ and $B$ related by $Ak-B=0$, where $k=\sqrt{|c-C|}$. From the equality $\<V,N\>=A\cos(a)$ we get $\cos(a)=0$ (since $A$ cannot vanish) and then $W_a(p)=\eta_2(p)$. Take $z_0\in\R{}$ such that $f(z_0/R)=kRg(z_0/R)$, and define a differentiable function $\varphi:P^{n-2}\to\R{n+1}_\nu$ by $\varphi(p)=\Psi_a(p,z_0)$. A straightforward computation yields
\[
d\varphi_p(v)=\left[f\Big(\frac{z_0}R\Big)-kRg\Big(\frac{z_0}R\Big)\right]v=0,
\]
for any $p\in P^{n-2}$ and $v\in T_pP^{n-2}$. Hence, $\varphi$ is a constant $q_0\in\R{n+1}_\nu$ and this shows that $M$ is a conical hypersurface with vertex at $q_0$.
\end{proof}

\section{Concircular helices in $\M^3(C)$}
\label{s:helices}

We begin this section by showing a couple of examples of concircular helices in $\M^3(C)$. Let $\gamma:I\to\M^3(C)$ be an arclength parametrized curve satisfying the following Frenet-Serret equations:
\begin{align}
T_\gamma'(s)=\nabla^0_{T_\gamma}T_\gamma(s)&= -C\gamma(s)+\kappa_\gamma(s)\,N_\gamma(s),\nonumber\\
N_\gamma'(s)=\nabla^0_{T_\gamma}N_\gamma(s) &= -\kappa_\gamma(s)\,T_\gamma(s)+\tau_\gamma(s)\,B_\gamma(s),\label{FS-eq0}\\
B_\gamma'(s)=\nabla^0_{T_\gamma}B_\gamma(s) &= -\tau_\gamma(s)\,N_\gamma(s),\nonumber
\end{align}
where $\nabla^0$ denotes the Levi-Civita connection on $\R{4}_\nu$. As usual, $\kappa_\gamma$ and $\tau_\gamma$ are called the curvature and torsion of $\gamma$.

\begin{exam}\label{ex-plana}
\emph{Planar curves}, i.e. curves $\gamma$ with zero torsion. These curves live in a surface $M^2(C)$ totally geodesic in $\M^3(C)$. That means there is a constant vector $\pc$ in $\R{4}_\nu$ such that $M^2(C)=\M^3(C)\cap H_0(\pc)$, where $H_0(\pc)$ is the hyperplane through the origin orthogonal to $\pc$. Since $\<N_\gamma, V\>=0$, for $V=\pc -C\<\pc,\phi\>\phi$, $\gamma$ is a concircular helix.
\end{exam}

\begin{exam}\label{ex-rectif}
\emph{Rectifying curves}, i.e. curves for which there exists a point $\pc $ in $\M^3(C)$ such that the geodesics connecting $\pc $ with $\gamma(s)$ are orthogonal to the principal normal geodesic starting from $\gamma(s)$, see \cite{LO14a,LO15}. The nonplanar rectifying curves are characterized by the condition $\<N_\gamma,\pc \>=0$. Then $\gamma$ is a concircular helix since $\<N_\gamma,V\>=0$ for the concircular vector field $V=\pc -C\<\pc ,\phi\>\phi$.
\end{exam}

In the following result we show that the restriction to a curve $\gamma$ of a concircular vector field $V$ is a vector field along $\gamma$ satisfying a property similar to that of concircular vector fields. For this reason, such a vector field will be called a \emph{concircular vector field along a curve}.

\begin{prop}\label{Vlocal}
Let $\gamma:I\to\M^n(C)$ be a differentiable curve and consider a vector field $\bfv$ along $\gamma$. Then $\bfv$ is the restriction to $\gamma$ of a concircular vector field $V$ on $\M^n(C)$ if and only if $\frac{\ol{D}\bfv}{dt}=\omega\,T_\gamma$, where $\omega:I\to\R{}$ is a differentiable function with $\omega'=-C\<\bfv,T_\gamma\>$, $\ol D$ being the covariant derivative along $\gamma$.
\end{prop}
\begin{proof}
Let us assume that $V$ is an extension of $\bfv$ to $\M^n(C)$ which is a concircular vector field. Then there exists a constant vector $\pc \in\R{n+1}_\nu$ such that $\pc =V+C\<\pc ,\phi\>\phi$. If $X$ is a local extension of the tangent vector $T_\gamma$ then
\[
\frac{\ol D\bfv}{dt}=\bnabla_XV|_\gamma=(\mu X)|_\gamma=\omega T_\gamma,
\]
where the differentiable function $\omega:I\to\R{}$ is the restriction of $\mu$ along $\gamma$. On the other hand, from theorem \ref{t1} we have $\mu=-C\<\pc ,\phi\>$, and then $\omega=-C\<\pc ,\gamma\>$ and $\omega'=-C\<\bfv,T_\gamma\>$.

To prove the converse, let us consider the vector field $Y$ along $\gamma$ given by $Y=\bfv-\omega\gamma$. By derivating here with the Euclidean derivative, we obtain that $Y$ is a constant vector $\pc \in\r{n+1}\nu$ along $\gamma$. By defining  $V=p_0-C\<p_0,\phi\>\phi$ and bearing theorem \ref{t1} in mind we deduce $\bfv$ is the restriction to $\gamma$ of a concircular vector field.
\end{proof}
\bigskip

For simplicity, and since $\omega(s)=\mu(\gamma(s))$, in what follows we will use $\mu(s)$ instead of $\omega(s)$. In the following, we will characterize the concircular helices in $\M^3(C)$. Note that $\gamma$ is a concircular helix if and only if there exists a point $p_0\in\R4_\nu$ such that $\<N_\gamma,p_0\>=\lambda$ constant.

Since the case $\lambda=0$ reduces to planar or rectifying curves, we will exclude this case from our study. A concircular helix $\gamma$ (with axis $V$) in $\M^3(C)$ is said to be \emph{proper} if $\gamma$ is a nonplanar curve with $\lambda\neq0$, $\lambda$ being the constant function $\<N_\gamma,V\>$.

\def\a{t}
\def\b{z}
Let $\gamma(s)\subset\M^3(C)\subset\R{4}_\nu$ be an arclength parametrized concircular helix, and suppose it is a proper one. From proposition \ref{Vlocal} we can write
\begin{equation}\label{YRef}
\bfv(s)=V(\gamma(s))=\a(s)\,T_\gamma(s)+\lambda\,N_\gamma(s)+\b(s)\,B_\gamma(s),
\end{equation}
for certain differentiable functions $\a$ and $\b$. To simplify the writing we will eliminate the $s$ parameter. By derivating in (\ref{YRef}) we get
\[
\mu\,T_\gamma=(\a'-\lambda\kappa_\gamma)\,T_\gamma+ (\a\kappa_\gamma-\b\tau_\gamma)\,N_\gamma+ (\b'+\lambda\tau_\gamma)\,B_\gamma,
\]
and then
\begin{equation}\label{3cond}
\a'-\lambda\kappa_\gamma =\mu,\quad
\a\kappa_\gamma-\b\tau_\gamma =0,\quad
\b'+\lambda\tau_\gamma =0.
\end{equation}
Note that $\b\neq0$ since $\gamma$ is a proper concircular helix. Now we distinguish two cases, according to $\a/\b$ (called the \emph{rectifying slope} of $\gamma$) is a constant or not.

\noindent\textbf{Case 1:} the rectifying slope $\a/\b$ of $\gamma$ is a constant function. Then the Lancret curvature $\rho=\tau_\gamma/\kappa_\gamma$ is constant, and by using the first and third equations of (\ref{3cond}) we get
\begin{equation}\label{eqcurvlancret}
\kappa_\gamma=\frac{-1}{\lambda(1+\rho^2)}\mu\qquad\text{and}\qquad \tau_\gamma=\rho\kappa_\gamma,
\end{equation}
where the function $\mu$ satisfies
\begin{equation}\label{eqmuhg}
\mu''+C\frac{\rho^2}{1+\rho^2}\mu=0.
\end{equation}
Now we will show that equations (\ref{eqcurvlancret}) and (\ref{eqmuhg}) characterize proper concircular helices with constant rectifying slope. Let $\gamma$ be a curve satisfying these two equations. Define a function $\b$ by
\[
\b=-\frac{\mu'}{C\rho},
\]
and consider the vector field $\bfv=\b\,D_\gamma+ \lambda N_\gamma$, where $D_\gamma=\rho T_\g+B_\g$ is the modified Darboux vector and $\lambda=-1/(m(1+\rho^2))$. Then
\[
\bnabla_{T_\gamma}\bfv=\frac{\rho}{1+\rho^2}\mu\,(\rho T_\gamma+B_\gamma)-\lambda m \mu T_\gamma+\lambda \rho m \mu B_\gamma=\mu T_\gamma.
\]
Since $\mu'=-C\<\bfv,T_\gamma\>$, from proposition \ref{Vlocal} we get $\gamma$ is a concircular helix.

Hence we have shown the following result.
\begin{prop}\label{prophgm3}
Let $\gamma$ be an arclength parametrized nonplanar curve in $\M^3(C)$. Then $\gamma$ is a proper concircular helix with constant rectifying slope if and only if its curvature and torsion are given by $\kappa_\gamma=m\mu$ and $\tau_\gamma=\rho\kappa_\gamma$, where $m$ and $\rho$ are nonzero constants and the function $\mu$ satisfies (\ref{eqmuhg}).
\end{prop}

\noindent\textbf{Case 2:} the rectifying slope $\a/\b$ is a nonconstant function.
From the first equation of (\ref{3cond}), bearing in mind that $\mu'=-C\<\bfv,T_\g\>$, we have
\begin{equation}
\mu''+C\mu=-C\lambda\kappa_\gamma,\label{mu''curvas1}
\end{equation}
and then the second and third equations of (\ref{3cond}) lead to
\begin{equation}
\Big(\frac{\mu'}{\rho}\Big)'=C\lambda\tau_\gamma.\label{mu''curvas2}
\end{equation}

As before, we will show that equations (\ref{mu''curvas1}) and (\ref{mu''curvas2}) characterize concircular helices (when $\rho$ is a nonconstant function). Let $\gamma$ be an arclength parametrized curve satisfying (\ref{mu''curvas1}) and (\ref{mu''curvas2}), for a constant $\lambda$ and a differentiable function $\mu$.
Define the functions
\begin{equation}\label{ab}
\a=-\frac{\mu'}{C}\qquad\text{and}\qquad \b=\frac{\a}{\rho},
\end{equation}
and consider the vector field $\bfv$ along $\gamma$ given by
\begin{equation}\label{Darboux}
\bfv=\b D_\g+\lambda N_\g=\a\,T_\gamma+\lambda\,N_\gamma+\b\,B_\gamma.
\end{equation}

From (\ref{mu''curvas1}) we get $\a'-\lambda\kappa_\gamma=\mu$, and from (\ref{mu''curvas2}) we obtain $\b'=-\lambda\tau_\gamma$. Then by derivating in (\ref{Darboux}) we have
\begin{equation*}
\bnabla_{T_\gamma}\bfv=  (\a'-\lambda\kappa_\gamma)T_\gamma+(\a\kappa_\gamma-\b\tau_\gamma)N_\gamma+(\lambda\tau_\gamma+\b')B_\gamma=\mu\,T_\gamma.
\end{equation*}
Therefore, since $\mu'=-C\<\bfv,T_\gamma\>$, from proposition \ref{Vlocal} we deduce $\gamma$ is a concircular helix.

In conclusion, putting cases 1 and 2 together, we have proved the following result. Note that equations (\ref{mu''curvas1}) and (\ref{mu''curvas2}) are equivalent to (\ref{eqcurvlancret}) and (\ref{eqmuhg}) in the case $\rho$ constant.

\begin{theo}\label{teoconc1}
Let $\gamma$ be an arclength parametrized nonplanar curve in $\M^3(C)$. Then $\g$ is a proper concircular helix if and only if equations (\ref{mu''curvas1}) and (\ref{mu''curvas2}) are satisfied, for a constant $\lambda\in\R{}$ and a differentiable function $\mu$. Moreover, the axis $V$ of $\gamma$ is the extension of the vector field $\bfv$ given in (\ref{Darboux}), $a$ and $b$ being the differentiable functions given in  (\ref{ab}).
\end{theo}

\section{Geodesics of concircular surfaces}
\label{s:geodesics}

Let $M$ be a nontrivial concircular surface in $\M^3(C)$ with axis $V$, and let us consider $\g(s)$ an arclength parametrized geodesic of $M$. From theorem \ref{sctheo}, we can assume that $\gamma$ is locally written as $X(t(s),z(s))$, where $X$ is the parametrization (\ref{teoparametrization}). Since the principal normal vector field $N_\g$ of $\gamma$ is collinear with
the unit normal vector field $N$ of $M$ in $\M^3(C)$, then $\<N_\gamma,V|_\gamma\>$ is constant, and then $\g$ is a concircular helix in $\M^3(C)$. The goal of this section is to prove the converse.

First, we are going to obtain the equations of geodesics in a concircular surface $M$. Let $\gamma(s)=X\big(t(s),z(s)\big)$ be an arclength parametrized geodesic of $M$, with $\kappa_\gamma>0$. Then $T_\gamma(s)=t'(s)X_t(t(s),z(s))+z'(s)X_z(t(s),z(s))$ and so there exists a differentiable funcion $\theta$ such that
\begin{align}
t'(s)\sqrt{E(t(s),z(s))} &= \sin\theta(s),\label{p1.1}\\
z'(s) &= \cos\theta(s).\label{p1.2}
\end{align}
Hence $T_\gamma(s)=\sin\theta(s)\,T_\delta(t(s))+\cos\theta(s)\,X_z(t(s),z(s))$, and by taking derivative here we get
\begin{align}\label{p1.3}
-C\gamma(s)+\kappa_\gamma(s)\,N_\gamma(s)=&\;\theta'(s)\big(\cos\theta(s) T_\delta(t(s))-\sin\theta(s) X_z(t(s),z(s))\big)+\nonumber\\
  & \;\sin\theta(s) t'(s)\big(k\eta(t(s))-C\delta(t(s))+\kappa_\delta(t(s)) N_\delta(t(s))\big)+\\
  & \;\cos\theta(s) \big(t'(s)X_{tz}(t(s),z(s))+ z'(s)X_{zz}(t(s),z(s))\big).\nonumber
\end{align}
Bearing in mind that $\{ T_\delta,\,X_z,\,N,\,\frac{1}{R} \gamma\}$ is an orthonormal frame of $\R4_\nu$ along $\g$, we have
\begin{align*}
\delta(t(s))= & \ts f\big(\frac{z(s)}{R}\big)\gamma(s)-R g\big(\frac{z(s)}{R}\big) X_z(t(s),z(s)),\\
N_\delta(t(s))= & \ts -\sin(a) N(t(s),z(s))+\cos(a)\Big(f\big(\frac{z(s)}{R}\big)X_z(t(s),z(s))+ \frac{\e}{R}g\big(\frac{z(s)}{R}\big)\gamma(s)\Big),\\
\eta(t(s))=& \ts \cos(a) N(t(s),z(s))+\sin(a)\Big(f\big(\frac{z(s)}{R}\big)X_z(t(s),z(s))+ \frac{\e}{R}g\big(\frac{z(s)}{R}\big)\gamma(s)\Big),\\
X_{tz}(t(s),z(s))= & \ts -\Big(\frac{\e}{R}g\big(\frac{z(s)}{R}\big)+ f\big(\frac{z(s)}{R}\big)\big(\sin(a)k+\cos(a)\kappa_\delta(t(s))\big)\Big)T_\delta(t(s)), \\
X_{zz}(t(s),z(s))= & -C\gamma(s).
\end{align*}
From these equations, jointly with (\ref{p1.3}) and the fact that $\gamma$ is a geodesic in $M$ (and so  $N_\gamma(s)=N(\g(s))$; the case $N_\gamma(s)=-N(\g(s))$ is similar), we deduce
\begin{align}
\theta'(s)&=\ts t'(s)\Big(CR g\big(\frac{z(s)}{R}\big)+ f\big(\frac{z(s)}{R}\big)\big(\sin(a)k+\cos(a)\kappa_\delta(t(s))\big)\Big),\label{eq1}\\
\kappa_\gamma(s)&=\sin\theta(s) t'(s)\big(k\cos(a)-\sin(a)\kappa_\delta(t(s))\big).\label{eq2}
\end{align}
On the other hand, it is easy to see that $B_\gamma(s)=\cos\theta(s) T_\delta(t(s))-\sin\theta(s) X_z(t(s),z(s))$ and by taking derivative here we have
\begin{equation}\label{tor}
  \tau_\gamma(s)=\cos\theta(s) t'(s)\big(-\cos(a)k+\sin(a)\kappa_\delta(t(s))\big).
\end{equation}
Hence, we have shown the following result.
\begin{prop}\label{intp1}
Let $M$ be a nontrivial concircular surface in $\M^3(C)$, locally parametrized by (\ref{teoparametrization}). An arclength parametrized curve $\gamma(s)=X\big(t(s),z(s)\big)$, with $\kappa_\gamma>0$, is a geodesic if and only if there is a differentiable funcion $\theta(s)$ such that equations (\ref{p1.1}), (\ref{p1.2}) and (\ref{eq1}) are satisfied. Moreover, the curvature and torsion of $\gamma$ are given by (\ref{eq2}) and (\ref{tor}), respectively.
\end{prop}

We finish this section with the following characterization of concircular helices in $\M^3(C)$.
\begin{theo}\label{th1}
Let $\gamma(s)$ be an arclength parametrized curve in $\M^3(C)$, $\kappa_\g>0$. Then $\g$ is a proper concircular helix if and only if $\gamma$ is (congruent to) a geodesic of a proper concircular surface.
\end{theo}
\begin{proof}
We need only prove the direct implication.
Let $\gamma(s)$ be a arclength parametrized proper concircular helix in $\M^3(C)$ with axis $V=\pc-C\<\pc ,\phi\>\phi$, such that $\<N_\gamma,V\>=\lambda$ is constant along $\g$. Let $M$ be the ruled surface with base curve $\gamma$ and director curve $D_\g$ (the unit Darboux vector field of $\gamma$), which can be parametrized as follows
\begin{equation}\label{superfgamma}
X(s,z)=f\Big({\frac{z}{R}}\Big)\gamma(s)+R\,g\Big({\frac{z}{R}}\Big)\, \left(\frac{\rho(s)}{\sqrt{1+\rho(s)^2}}T_\gamma(s)+\frac{1}{\sqrt{1+\rho(s)^2}}B_\gamma(s)\right).
\end{equation}
Since $X_s\in\text{span}\{T_\gamma,B_\g\}$ and $X_z\in\text{span}\{\gamma,\,T_\gamma,\,B_\gamma\}$, we obtain that the unit normal vector field $N$ is collinear with the principal normal vector field $N_\gamma$. From here we conclude that $M$ is a concircular surface in $\M^3(C)$ with axis $V$, and that $\g$ is a geodesic in $M$. This concludes the proof.
\end{proof}

\section*{Acknowledgements}

This research is part of the grant PID2021-124157NB-I00, funded by MCIN/ AEI/ 10.13039/ 501100011033/ ``ERDF A way of making Europe''. Also supported by ``Ayudas a proyectos para el desarrollo de investigación científica y técnica por grupos competitivos'', included in the ``Programa Regional de Fomento de la Investigación Científica y Técnica (Plan de Actuación 2022)'' of the Fundación Séneca-Agencia de Ciencia y Tecnología de la Región de Murcia, Ref. 21899/PI/22.

\end{document}